\documentclass[a4paper,oneside,11pt, reqno]{amsart}

\usepackage{mathrsfs}
\usepackage{amsfonts,amssymb,amsmath,amsthm}
\usepackage{url}
\usepackage{enumerate}

\usepackage{bbm}
\usepackage{float}
\usepackage{pstricks}
\usepackage{amscd}
\usepackage{amsmath}
\usepackage{amsxtra}
\usepackage[T1]{fontenc}
\usepackage[utf8]{inputenc}
\usepackage{lipsum}
\usepackage{tikz}
\usetikzlibrary{arrows}
\usetikzlibrary{calc}
\usepackage{hyperref}

\theoremstyle{plain}
\newtheorem{theorem}{Theorem}[section]

\newtheorem{lemma}[theorem]{Lemma}
\newtheorem{proposition}[theorem]{Proposition}
\newtheorem{corollary}[theorem]{Corollary}

\theoremstyle{definition}

\theoremstyle{remark}
\newtheorem{remark}[theorem]{Remark}

\newtheorem{case}{Case}[section]

\newtheorem{casen}{Case}
\newtheorem{case-new}{Case}
\newtheorem{subcase}{Subcase}

\numberwithin{equation}{section}

\newcommand{\mf}{\mathfrak}
\newcommand*{\Ge}{\geqslant}
\newcommand*{\Le}{\leqslant}

\usepackage{etoolbox}

\makeatletter
\@namedef{subjclassname@2020}{%
  \textup{2020} Mathematics Subject Classification}
\makeatother

\usepackage{mathrsfs}
\usepackage{epsfig}
\usepackage[active]{srcltx}

\newcommand{\ncom}{\newcommand}
\ncom{\bq}{\begin{equation}}
\ncom{\eq}{\end{equation}}
\ncom{\beqn}{\begin{eqnarray*}}
\ncom{\eeqn}{\end{eqnarray*}}
\ncom{\beq}{\begin{eqnarray}}
\ncom{\eeq}{\end{eqnarray}}
\ncom{\nno}{\nonumber}
\ncom{\rar}{\rightarrow}
\ncom{\Rar}{\Rightarrow}
\ncom{\noin}{\noindent}
\ncom{\bc}{\begin{centre}}
\ncom{\ec}{\end{centre}}
\ncom{\sz}{\scriptsize}
\ncom{\rf}{\ref}
\ncom{\sgm}{\sigma}
\ncom{\Sgm}{\Sigma}
\ncom{\dt}{\delta}
\ncom{\Dt}{Delta}
\ncom{\lmd}{\lambda}
\ncom{\Lmd}{\Lambda}
\ncom{\eps}{\epsilon}
\ncom{\pcc}{\stackrel{P}{>}}
\ncom{\dist}{{\rm\,dist}}
\ncom{\sspan}{{\rm\,span}}
\ncom{\im}{{\rm Im\,}}
\ncom{\sgn}{{\rm sgn\,}}
\ncom{\ba}{\begin{array}}
\ncom{\ea}{\end{array}}
\ncom{\eop}{\hfill{{\rule{2.5mm}{2.5mm}}}}
\ncom{\eoe}{\hfill{{\rule{1.5mm}{1.5mm}}}}
\ncom{\eof}{\hfill{{\rule{1.5mm}{1.5mm}}}}
\ncom{\hone}{\mbox{\hspace{1em}}}
\ncom{\htwo}{\mbox{\hspace{2em}}}
\ncom{\hthree}{\mbox{\hspace{3em}}}
\ncom{\hfour}{\mbox{\hspace{4em}}}
\ncom{\hsev}{\mbox{\hspace{7em}}}
\ncom{\vone}{\vskip 2ex}
\ncom{\vtwo}{\vskip 4ex}
\ncom{\vonee}{\vskip 1.5ex}
\ncom{\vthree}{\vskip 6ex}
\ncom{\vfour}{\vspace*{8ex}}
\ncom{\norm}{\|\;\;\|}
\ncom{\integ}[4]{\int_{#1}^{#2}\,{#3}\,d{#4}}
\ncom{\inp}[2]{\langle{#1},\,{#2} \rangle}
\ncom{\Inp}[2]{\Langle{#1},\,{#2} \Langle}
\ncom{\vspan}[1]{{{\rm\,span}\#1 \}}}
\ncom{\dm}[1]{\displaystyle {#1}}

\keywords{Joint completely monotone, Bessel function, toral $m$-isometry, Cauchy dual}

\subjclass[2020]{Primary 44A60, 47A13; Secondary 47B37, 47B39}

\begin{document}
\title[ The Cauchy dual subnormality problem]{A complete solution to the Cauchy dual subnormality problem for torally expansive toral $3$-isometric weighted $2$-shifts}
\author[R. Nailwal]{Rajkamal Nailwal}
\address{Department of Mathematics and Statistics\\
Indian Institute of Technology Kanpur, India}
    \email{rnailwal@iitk.ac.in, raj1994nailwal@gmail.com}
\maketitle
\begin{abstract} 
In this paper, we present a complete solution to the Cauchy dual subnormality problem for torally expansive toral $3$-isometric weighted $2$-shifts. This solution is obtained by solving a couple of Hausdorff moment problems arising from $2$-variable polynomials of lower bi-degree.
\end{abstract}

\section{Introduction}
The Cauchy dual subnormality problem (for short CDSP) in $n$-variables asks whether
the Cauchy dual of an $m$-isometric $n$-tuple is jointly subnormal.
This problem for $2$-isometries has received significant attention, with extensive studies revealing intriguing links to moment theory and complex analysis (see \cite{ACJS, BS2019, CGR}; for solutions to CDSP for various classes of $m$-isometries, see \cite[Proposition~1.7]{AC2017}, \cite[Proposition~1.3]{ACJS},\cite[Theorem~3.4]{ BS2019}, \cite[Theorem~2.1]{CGR},
\cite[Theorem~4.6]{JK2021} etc; for Brownian-type operators, see \cite[Theorem~1.2]{SZJS2021}). In this paper, we present a complete solution to the Cauchy dual subnormality problem for torally expansive toral $3$-isometric weighted $2$-shifts. A special case of this (the case of separate 2-isometries) has been obtained in \cite[Theorem~4.9]{ACN}. 
Moreover, we present several families of Hausdorff moment net arising from the reciprocal of polynomials of bi-degree $(2, 1)$ and $(2,
2)$. Before we state the main result of this paper, let us fix some notations and recall the relevant notions.

Let $n$ be a positive integer and $X$ be a set.  The notation $X^n$ represents the Cartesian product of $X$ with itself, taken $n$ times. Denote by $\mathbb{Z}_+$ and $\mathbb{R}_+,$ the set of nonnegative integers and nonnegative real numbers, respectively. Let $\alpha = (\alpha_1, \ldots, \alpha_n), \beta
= (\beta_1, \ldots, \beta_n) \in \mathbb{Z}^n_+.$ Set $|\alpha|=\alpha_1 + \cdots + \alpha_n$ and  $(\beta)_\alpha = \prod_{j=1}^n (\beta_j)_{\alpha_j},$ where $(\beta_j)_{0}=1,$ $(\beta_j)_{1}=\beta_j$ and 
 $$(\beta_j)_{\alpha_j} = \beta_j(\beta_j -1) \cdots (\beta_j -\alpha_j+1), \quad \alpha_j \Ge 2, ~j=1, \ldots, n.$$ 
We denote
$\alpha \Le \beta$ if $\alpha_j \Le \beta_j$ for every $j=1, \ldots, n.$
For $\alpha \Le \beta$, we let $\binom{\beta}{\alpha}=\prod_{j=1}^n \binom{\beta_j}{\alpha_j}.$

For a net $\{a_\alpha\}_{\alpha \in \mathbb Z^n_+}$ and 
$j=1, \ldots, n,$ let $\triangle_j$ denote the {\it forward difference operator} given by
\beqn
\triangle_j a_\alpha = a_{\alpha + \varepsilon_j} - a_\alpha, \quad \alpha \in \mathbb Z^n_+,
\eeqn
where $\varepsilon_j$ stands for the $n$-tuple with $j$th entry equal to $1$ and $0$ elsewhere. Note that for any $i,j \in \{1,2,\ldots, n\}, 
\triangle_i \triangle_j =\triangle_j \triangle_i.$
For $\alpha=(\alpha_1, \ldots, \alpha_n) \in \mathbb Z^n_+,$ let $\triangle^\alpha$ denote the operator $\prod_{j=1}^n \triangle^{\alpha_j}_j.$  
For a polynomial $p$ in one variable, let $\deg p$ denote the degree of $p.$ A polynomial $p$ of two variables is said to be of {\it bi-degree} $\alpha=(\alpha_1, \alpha_2) \in \mathbb Z^2_+$ if for each $j=1, 2,$ $\alpha_j$ is the largest integer for which $\partial^{\alpha_j}_j p \neq 0.$

We now recall the definition of joint complete monotonicity of a net.
A net $\mf a = \{a_\alpha\}_{\alpha \in \mathbb{Z}^n_+}$ is said to be {\it joint completely monotone} if
 \beqn
 (-1)^{|\beta|} \triangle^{\beta} a_{\alpha} \Ge 0,  \quad \alpha, \beta \in \mathbb Z^n_+.
 \eeqn
When $n=1,$ we simply refer to $\mf a$ as a {\it completely monotone} sequence. We say $\mf a$ is a {\it separate completely monotone} if  for every $j\in \{1, \ldots, n\}$ ,  $k \in \mathbb{Z_+},$
\beqn
 (-1)^{k} \triangle_j^{k} a_{\alpha} \Ge 0,  \quad  \alpha \in \mathbb Z^n_+.
\eeqn
For a detailed account of complete monotonicity in one and several variables, the reader is referred to \cite{Kball, BCR1984, BD2005, RZ2019}.
\begin{remark}\label{remark-cm-gen} 
It is readily seen that a  joint completely monotone net is separate completely monotone.
Also, if $\phi$ is completely monotone function on $\mathbb{Z}_+^n,$ then for any $\beta \in \mathbb{Z}_+^n,$ the function $\alpha \mapsto \phi(\alpha+\beta)$ is also completely monotone on $\mathbb{Z}_+^n.$ 
\hfill $\diamondsuit$
\end{remark}
We now recall a solution to the multi-dimensional Hausdorff moment problem. A net $\mf a = \{a_\alpha\}_{\alpha \in \mathbb{Z}^n_+}$ is joint completely monotone if and only if it is a {\it Hausdorff moment net}, that is, if there exists a positive Radon measure $\mu$ concentrated on $[0,1]^n$ such that 
\beqn
a_\alpha =\int _{[0,1]^n} t^{\alpha} \mu(dt), \quad \alpha \in \mathbb Z^n_+.
\eeqn 
(see \cite[Proposition~4.6.11]{BCR1984}). If such a measure $\mu$ exists, then it is unique. This is a consequence of the $n$-dimensional Weierstrass theorem and the Riesz representation theorem (see  \cite[Theorem~2.14]{R2006} and \cite[Lemma 4.11.3]{S2015}). 
We refer to $\mu$ as the {\it representing measure} of $\mf a.$

We now recall some operator-theoretic prerequisite. Let $n$ be a positive integer. A operator tuple $T=(T_1, \ldots, T_n)$ on a complex separable Hilbert space $H$ is said to be {\it commuting n-tuple} if  $T_1, \ldots, T_n$ are bounded linear operator on $H$ and $T_i T_j = T_j T_i$ for every $1 \Le i \neq j \Le n.$
  A commuting $n$-tuple $T$ is said to be a {\it toral expansion} (resp. a {\it toral contraction}) if $T^*_jT_j \Ge I$ (resp. $T^*_jT_j \Le I$) for every $j \in \{1, \ldots, n\}.$  We say that a commuting $n$-tuple $T=(T_1, \ldots, T_n)$ is {\it jointly subnormal} if there exist a Hilbert space $K$ containing $H$ and a commuting $n$-tuple $N$ of normal operators $N_1, \ldots, N_n$ on $K$ such that
\beqn
T_j = {N_j}|_{H}, \quad j=1, \ldots, n.
\eeqn
Let $m$ be a positive integer. Following \cite{Ag1990, Athavale2001, R1988}, we say that a commuting $n$-tuple $T$ is said to be a {\it toral $m$-isometry} if 
\beqn 
\sum_{\overset{\alpha \in \mathbb Z^n_+}{0 \Le \alpha \Le \beta}} (-1)^{|\alpha|} \binom{\beta}{\alpha}T^{*\alpha}T^{\alpha} = 0, \quad \beta \in \mathbb Z^n_+, ~|\beta|=m,
\eeqn
where $T^{\alpha}$ denotes the bounded linear operator $\prod_{j=1}^nT^{\alpha_j}_j$ and $T^{*\alpha}$ stands for the Hilbert space adjoint of $T^{\alpha}.$ The reader is referred to  \cite{Ag1990, AS1995, Athavale2001, AS1999, CC2012, JJS2020, S2001} for the basic theory of toral $m$-isometries.
 
Assume that $T^*_jT_j$ is invertible for every $j=1, \ldots, n.$ Following \cite{CC2012, S2001}, we refer to the $n$-tuple $T^{\mathfrak{t}}:=
(T^{\mathfrak{t}}_1, \cdots, T^{\mathfrak{t}}_n)$ as the \textit{operator tuple torally Cauchy dual} to $T$ where $T^{\mathfrak{t}}_j :=
T_j(T^*_jT_j)^{-1},$ for $j=1, \cdots, n.$ Note that {\it toral $m$-isometric tuple} $T=(T_1, \ldots, T_n)$ is a {\it separate $m$-isometric tuple}, that is, $T_1, \ldots, T_n$ are $m$-isometries. By \cite[Lemma~1.21]{AS1995}, $T_{j}$ is left invertible for $1 \Le j \Le n.$ Hence, the toral Cauchy dual of a toral $m$-isometric $n$-tuple exists.
 
 Let $\mathscr H$ be a Hilbert space with orthonormal basis $\mathscr E = \{e_{\alpha}:\alpha \in \mathbb{Z}^n_+\}.$  Let ${\bf w} = \big\{w^{(j)}_{\alpha} : j=1, \ldots, n, ~\alpha \in \mathbb{Z}^n_+ \big\}$ be a collection of complex numbers. For $j=1, \ldots, n$ and any $\alpha \in \mathbb{Z}^n_+$, define $\mathscr W = (\mathscr W_1, \ldots, \mathscr W_n)$ by
  $$ \mathscr W_j e_\alpha = w^{(j)}_\alpha e_{\alpha + \varepsilon_j},$$ 
  where $\varepsilon_j$ is a vector with a $1$ in the $j$th position and zeros elsewhere. Note that by extending it linearly on $\mathscr E,$ $\mathscr W_1, \ldots, \mathscr W_n$ define bounded operator on $\mathscr H$  if and only if $\sup_{\alpha \in \mathbb Z^n_+}|w^{(j)}_\alpha| < \infty$ for every $j=1, \ldots, n.$
 Also for any $ i, j \in \{ 1, \ldots, n \}, \mathscr W_i$ and $\mathscr W_j$ commute if and only if
\beqn
 w^{(i)}_\alpha w^{(j)}_{\alpha
    + \varepsilon_i}=w^{(j)}_\alpha w^{(i)}_{\alpha + \varepsilon_j}, \quad \alpha \in {\mathbb Z}^n_+.
    \eeqn
    Let $\mathscr W$ be a commuting weighted $n$-shift. Note that for any $\beta \in \mathbb Z^n_+,$ there exists a positive scalar $m(\beta)$ such that
 \beq \label{W-beta}
 \mathscr W^{\beta} e_0= m(\beta)e_\beta.
 \eeq
  For more information on the basic theory of weighted multi-shifts, the reader is referred to \cite{ AS1999, JS2001, JL1979}.
  
In what follows, we assume that $\bf w$ forms a bounded subset of positive real numbers and $\mathscr W$ is a commuting $n$-tuple. We will denote the weighted $n$-shift $\mathscr W$ with weight multi-sequence ${\bf w}$ by $\mathscr W : {{w^{(j)}_\alpha}}$. 
 
 Let $\mathscr W : \{w^{(j)}_\alpha\}$ be a weighted $n$-shift such that $\mathscr W^*_j \mathscr W_j$ is invertible for each $j=1, \ldots, n$. The operator tuple $\mathscr W^{\mathfrak{t}}$ torally Cauchy dual to the weighted $n$-shift $\mathscr W$, satisfies the following relation:
\beq \label{dual-action} \mathscr W^{\mathfrak{t}}_j e_\alpha =
\frac{1}{w^{(j)}_\alpha} \, e_{\alpha + \varepsilon_j}, \quad j=1, \ldots, n. 
\eeq
It is now easy to see that:
 \beq \label{reciprocal}
 \|(\mathscr W^{\mathfrak{t}})^\alpha e_0\|^2  = \frac{1}{\|\mathscr W^{\alpha} e_0\|^2}, \quad \alpha  \in \mathbb Z^n_+.
 \eeq
 
 To state the main result, we find it convenient to  introduce the following notation: For $i,j \in \{0,1,2\},$
\beq \label{notation}
\rho_{ij}={\triangle_1^i\triangle_2^j(\|\mathscr W^{\alpha}e_0\|^2)|_{\alpha =0} },\,\, \rho_1=2\rho_{10}-\rho_{20},\,\, \rho_2=2\rho_{01}-{\rho_{02}}.
\eeq
We are now ready to state the main result of this paper. For the sake of completeness, we include the separate 2-isometry case as part (a) below (see \cite[Theorem~4.9]{ACN}).

\begin{theorem} \label{CDSP-T3I}
Let $\mathscr W : \{w^{(j)}_\alpha\}$ be a  torally expansive toral $3$-isometric weighted $2$-shift and let 
$\mathscr W^{\mathfrak{t}}$ be the operator tuple 
torally Cauchy dual to $\mathscr W$. Let $\rho_1, \rho_2$ and $\rho_{ij}, i,j \in \{0,1,2\}$ be as given in \eqref{notation}. The following statements holds$:$ 
\begin{itemize}
 \item[$\mathrm{(a)}$] Assume that $\mathscr W$ is a separate $2$-isometry. Then the operator tuple $\mathscr W^{\mathfrak{t}}$
 is jointly subnormal if and only if $$\rho_{11}\Le \rho_{10}\rho_{01}.$$
 \item[$\mathrm{(b)}$] Assume that $\mathscr W_1$ is not a $2$-isometry. Then the operator tuple $\mathscr W^{\mathfrak{t}}$
 is jointly subnormal if and only if $\rho_1>0, \rho_1^2\Ge 8\rho_{20},$ 
 and exactly any one of the following holds$:$
  \begin{enumerate}
     \item[$\mathrm{(i)}$] $\rho_{11}=0,\rho_{01}=0,{\rho_{02}}=0$,
     \item[$\mathrm{(ii)}$] $\rho_{11}>0,\rho_2>0, \rho_{11}^2\Ge\rho_{20}\rho_{02},$
      \beqn\label{simple-inq}\notag
     (\rho_{20}\rho_2-\rho_{11}\rho_1)^2  \Le (4\rho_{11}^2-{\rho_{20}\rho_{02}}) (\frac{\rho^2_1}{4}-2\rho_{20}).
     \eeqn 
\end{enumerate}
\item[$\mathrm{(c)}$] Assume that $\mathscr W_2$ is not a $2$-isometry. Then the operator tuple $\mathscr W^{\mathfrak{t}}$
 is jointly subnormal if and only if $\rho_2>0, \rho_2^2\Ge 8\rho_{02},$ 
 and exactly any one of the following holds$:$
  \begin{enumerate}
     \item[$\mathrm{(i)}$] $\rho_{11}=0,\rho_{10}=0,{\rho_{20}}=0$,
     \item[$\mathrm{(ii)}$] $\rho_{11}>0,\rho_1>0, \rho_{11}^2\Ge\rho_{20}\rho_{02},$
      \beqn 
     (\rho_{02}\rho_1-\rho_{11}\rho_2)^2  \Le (4\rho_{11}^2-{\rho_{02}\rho_{20}}) (\frac{\rho^2_2}{4}-2\rho_{02}).
     \eeqn
\end{enumerate}
 \end{itemize}
\end{theorem}

\subsection*{Plan of the paper}
In Section~2, we consider polynomial $p:\mathbb{R}_+^2 \rightarrow (0,\infty)$ of the form $p(x,y)=b_0(x+b_1)(x+b_2)+a_0(x+a_1)y,$ where $a_0, a_1, b_0, b_1, b_2 \in \mathbb{R},$  
 with $b_1\Le b_2$ and $a_0,a_1 \neq 0.$ We describe all polynomials $p$ for which $1/p$ is a joint completely monotone net (see Theorem \ref{bi-deg-2-1-second}). As a consequence of Theorem~\ref{bi-deg-2-1-second}, we obtain some necessary conditions  for the polynomial $q(x,y)=b_0(x+b_1)(x+b_2)+(a_1x+a_2)y,$ whose reciprocal is a joint completely monotone net (see Corollary~\ref{Imp-coro}). 
In Section~3, we consider the polynomial $p(x,y)=a(x)+b(x)y+y^2,$ where $a(x)=a_0(x+a_1)(x+a_2), b(x)=b_0(x+b_1)$, $a_0, a_1,a_2,b_0,b_1 \in \mathbb{R}$ with $a_1 \Le a_2.$ Under the assumption $p(m,n)>0,$ we characterize the joint complete monotonicity of $\{1/p(m,n)\}_{m,n \in \mathbb{Z}_+}$ (see Theorem \ref{MT}). Proof of this theorem is fairly long and requires several lemmas (see Lemmas \ref{pos-(2,2)}-\ref{Ext-cm}).
In Section~4, we provide a solution to the Cauchy dual subnormality problem for torally expansive toral $3$-isometric weighted $2$-shifts, which completes the proof of Theorem~\ref{CDSP-T3I}. Note that the proof of Theorem~\ref{CDSP-T3I} relies on Theorems~\ref{bi-deg-2-1-second} and \ref{MT} and a characterization of toral $3$-isometries (see Proposition~\ref{3-iso-wt-char}).

\section{A special case of bi-degree $(2,1)$}
In this section, we present a proof of Theorem~\ref{bi-deg-2-1-second}. The proof of the sufficiency part of this theorem is obtained in \cite[Theorem~3.6]{ACN}. Here, we obtain a proof of the necessity part.

Recall that for a positive real number $\nu,$ 
the {\it Bessel function $J_{\nu}(z)$ of the first kind of order $\nu$} is given by
\beqn
J_{\nu}(z)=\Big(\frac{z}{2}\Big)^{\nu} \sum_{k=0}^\infty \Big(\frac{-z^{2}}{4}\Big)^k\frac{1}{k!\Gamma(\nu+k+1)}, \quad z \in \mathbb C \setminus (-\infty, 0],
\eeqn
where $\Gamma$ denotes the Gamma function. 
 \begin{theorem}[Special case of bi-degree $(2, 1)$]  \label{bi-deg-2-1-second}
Let $p:\mathbb{R}_+^2 \rightarrow (0,\infty)$ be a polynomial given by $p(x,y)=b(x)+a(x)y,$ where $a(x)=a_0(x+a_1)$ and $b(x)=b_0(x+b_1)(x+b_2),a_0, a_1, b_0, b_1, b_2 \in \mathbb{R},$  
 with $b_1\Le b_2$ and $a_0,a_1 \neq 0.$ Then the net $\left\{ \frac{1}{p(m,n)}\right\}_{m,n \in \mathbb{Z}_+}$ is joint completely monotone if and only if $b_1 \Le a_1 \Le b_2.$
\end{theorem}
\begin{proof}
Since range of $p$ is contained in $(0,\infty)$ and $a_0,a_1 \neq 0,$ an elementary checking shows that $a_0,a_1,a_2,b_0,b_1 >0$ (see discussion prior to \cite[Proposition~3.2]{ACN}).  It was implicitly recorded in the proof of \cite[Theorem~3.6]{ACN} that for $m,n \in \mathbb{Z}_+$
\beqn 
\frac{1}{p(m,n)} &=&  \int_{[0, 1]^2}t^{n} s^m_1 \, \frac{(s_1/t^{c_0})^{a_1-1}t^{c_0(b_1+b_2-a_1)-1}}{a_0t^{c_0}}\\ 
 && \sum_{k=0}^{\infty}\frac{(-c_0c_1\log {(s_1/t^{c_0})}\log t)^k}{k!^2}\mathbbm{1}_{[0,t^{c_0}]}(s_1)ds_1dt,
\eeqn
where $c_0=b_0/a_0>0$ and $c_1=(a_1-b_2)(a_1-b_1)$.
So the weight function for the net $\left\{\frac{1}{p(m,n)}\right\}_{(m,n)\in \mathbb{Z}_+^2}$ is
\beq \notag
w(s,t)=\frac{(s/t^{c_0})^{a_1-1}t^{c_0(b_1+b_2-a_1)-1}}{a_0t^{c_0}} \sum_{k=0}^{\infty}\frac{(-c_2\log {(s/t^{c_0})}\log t)^k}{k!^2}\mathbbm{1}_{[0,t^{c_0}]}(s),
\eeq 
where $s,t \in (0,1)$ and $c_2=c_0c_1.$ Sufficiency part follows from \cite[Theorem~3.6]{ACN}. To prove the necessity part, assume that $a_1 \notin [b_1,b_2].$ We will show that $w(s,t)<0$ on some open set contained in $(0,1)^2$. Since by the pasting lemma, $w(s,t)$ is continuous on $(0,1)^2,$ it only require to show that $w(s,t)<0$ for some $s,t \in (0,1)$. It now suffices to check that
\beq \notag
 \sum_{k=0}^{\infty}\frac{(-c_2\log {(s/t^{c_0})}\log t)^k}{k!^2}\mathbbm{1}_{[0,t^{c_0}]}(s) < 0,
\eeq
for some  $s,t \in (0,1)$. Observe that $c_2=c_0c_1>0$. Take $t_0=1/2$ and $s_0=\frac{e^{-\frac{5}{c_2\log(2)}}}{2^{c_0}}<\frac{1}{2^{c_0}}$.
It is easy to see that
\beqn 
 \sum_{k=0}^{\infty}\frac{(-c_2\log {(s_0/t_0^{c_0})}\log t_0)^k}{k!^2}\mathbbm{1}_{[0,{t_0}^{c_0}]}(s_0)= \sum_{k=0}^{\infty}\frac{(-5)^k}{k!^2} = J_{0}(2\sqrt{5}) \approx -0.3268,
\eeqn
where $ J_{0}(x) $ is the Bessel function of the first kind of order $ 0 $.
This, together with the continuity of $w(s,t)$ on $(0,1)^2,$ implies that  $\left\{\frac{1}{p(m,n)}\right\}_{m,n\in \mathbb{Z}_+}$ is not a joint completely monotone net. Therefore, we have $b_1 \Le a_1 \Le b_2.$ This completes the proof. 
\end{proof}

The following lemma is stated for frequent use (for a variant, see \cite[Lemma~3.1]{ACJS}).
\begin{lemma} \label{freq-3}
    Let $p$ be a polynomial of degree $2$ given by $p(x)=a+bx+cx^2,$ where $a,b,c\in \mathbb{R}$ such that $p(n)\neq 0,n \in \mathbb{Z}_+.$ Then the sequence $\{{1}/{p(n)}\}_{n \in \mathbb Z_+}$ is completely monotone if and only if $a,b,c$ are positive real numbers and $p$ is reducible over $\mathbb R.$  
\end{lemma}
\begin{proof}
To see the proof of the necessity part, assume that the sequence $\{{1}/{p(n)}\}_{n \in \mathbb Z_+}$ is completely monotone. Note that $p(n)>0, n\in \mathbb{Z}_+,$ and hence $a >0,$ $c>0.$ An application of \cite[Theorem~1.5]{AC17} shows that 
\beq \label{roots-l-half}
\mbox{the roots of $p$ lies in $\{z \in \mathbb C : \Re(z) < 0\}.$}
\eeq
Let, if possible, $p$ be irreducible. Since $p(0)>0,$ we must have $p(x)>0$ for all $x \in \mathbb{R}.$ 
An application of \cite[Propositions 4.3]{AC2017} together with \eqref{roots-l-half} shows that   $\{{1}/{p(n)}\}_{n \in \mathbb Z_+}$ is not completely monotone. This contradiction shows that  $p$ is reducible over $\mathbb R.$ Thus $p$ has negative real roots, say, $\alpha_1$ and $\alpha_2.$ Since $b=-c(\alpha_1+\alpha_2),$ $b$ is positive. For the proof of the sufficiency part, note that $\{{1}/{p(n)}\}_{n \in \mathbb Z_+}$ is product of two completely monotone sequence and hence the sequence $\{{1}/{p(n)}\}_{n \in \mathbb Z_+}$ is completely monotone.
\end{proof}

{ With Lemma~\ref{freq-3}, we can now obtain the following corollary. 
\begin{corollary}\label{Imp-coro}
    Let $q$ be a polynomial given by $q(x,y)=b(x)+a(x)y,$ where $b(x)=b_0(x+b_1)(x+b_2),$ $ a(x)=a_1x+a_2$, $a_1,a_2,b_0,b_1,b_2 \in \mathbb{R}$ with $b_1 \Le b_2$ such that $q(m,n)\neq 0, m,n \in \mathbb Z_+.$ 
    Then the following holds$:$
    \begin{enumerate}
        \item[$\mathrm{(i)}$] if $q(m,n)>0,m,n \in \mathbb{Z}_+,$ then $a_1,a_2\Ge0,$
        \item[$\mathrm{(ii)}$] if $\left\{\frac{1}{q(m,n)}\right\}_{m,n \in \mathbb Z_+}$ is a joint complete monotone net then $b_0,b_1,$ $b_2>0$ and $a_1,a_2\Ge 0.$  Moreover, $a_1$ and $a_2$ are zero or positive real numbers simultaneously.
    \end{enumerate}
\end{corollary} 
\begin{proof} Assume that $q(m,n)>0, m,n \in \mathbb{Z}_+.$ Let if possible $a_2<0.$ Choose a large value $n_0\in \mathbb Z_+$ such that $q(0,n_0)<0.$ This contradicts the assumption. Hence $a_2 \Ge 0.$ A similar argument can be used to see $a_1\Ge0.$ This completes the proof of (i). Assume that the net $\left\{\frac{1}{q(m,n)}\right\}_{m,n \in \mathbb Z_+}$ is joint completely monotone. Thus, it is separate completely monotone. This implies  $\left\{\frac{1}{q(m,0)}\right\}_{m \in \mathbb Z_+}$ is a completely monotone sequence. It now follows from $q(m,0) \neq 0, m \in \mathbb Z_+,$ and Lemma~\ref{freq-3}, that $b_0,b_1,b_2>0.$ Note that $q(m,n)>0, m,n \in \mathbb{Z}_+.$ By (i), we have $a_1,a_2 \Ge 0.$ We now consider two cases here.
\begin{case}
  $a_1=0$    
\end{case}
Let if possible $a_2> 0.$ In this case for large values of $n_0\in \mathbb Z_+,$ $q(.,n_0)$ is irreducible which contradicts the complete monotonicity of $\left\{\frac{1}{q(m,n_0)}\right\}_{m\in \mathbb Z_+}.$ Hence $a_2=0.$

\begin{case}
    $a_1>0$
\end{case}  Let if possible $a_2=0.$  By Remark~\ref{remark-cm-gen},  for $k \in \mathbb{Z}_+,$  $\left\{ \frac{1}{q(m+k,n)}\right\}_{m,n \in \mathbb{Z}_+}$ is a joint completely monotone net. By Theorem~\ref{bi-deg-2-1-second}, $b_1+k \Le k \Le b_2 +k.$ This yields $b_1 \Le 0,$ which is a contradiction. Hence $a_2 > 0.$
This completes the proof.
\end{proof}
\section{A special case of bi-degree $(2,2)$}
 In this section, we consider a class of polynomials of bi-degree $(2,2)$ and characterize the joint complete monotonicity of their reciprocals.
 
 \begin{theorem}[Special case of bi-degree $(2,2)$] \label{MT}
Let $p$ be a polynomial given by $p(x,y)=a(x)+b(x)y+y^2,$ where $a(x)=a_0(x+a_1)(x+a_2), b(x)=b_0(x+b_1)$, $a_0, a_1,a_2,b_0,b_1 \in \mathbb{R}$ with $a_1 \Le a_2$. Assume that $p(m,n)>0$ for every $m,n \in \mathbb{Z}_+.$ Then $\left\{ \frac{1}{p(m,n)}\right\}_{m,n \in \mathbb{Z}_+}$  is a joint completely monotone net if and only if $a_1,a_2,b_0,b_1>0$, $b_0^2 \Ge 4a_0,$
\beq \label{Impeq} 
a_0(a_2-a_1)^2\Le b_0^2(b_1-a_1)(a_2-b_1).
\eeq 
\end{theorem}
 
The following lemma plays an important role in solving CDSP for torally expansive toral $3$-isometric weighted $2$-shifts.
\begin{lemma} \label{pos-(2,2)}
    Let $p$ be a polynomial given by $p(x,y)=a(x)+b(x)y+y^2,$ where $a(x)=a_0(x+a_1)(x+a_2),$ $ b(x)=b_1x+b_2$, $a_0,a_1,a_2,b_1,b_2 \in \mathbb{R}$ with $a_1 \Le a_2$ such that $p(m,n)\neq 0, m,n \in \mathbb Z_+.$ Assume that $\left\{\frac{1}{p(m,n)}\right\}_{m,n \in \mathbb Z_+}$ is a joint complete monotone net. Then $a_0,a_1,a_2, b_1,b_2> 0.$
\end{lemma}
\begin{proof}
  A similar argument as used in the proof of Corollary~\ref{Imp-coro} shows that $a_0,a_1,a_2>0.$ By symmetry, one can see that $b_2>0.$ We now consider the following cases.
    \begin{case}
    $b_1=0$
    \end{case}
    Note that for large values of $m_0\in \mathbb{Z}_+,$  $p(m_0,.)$ is irreducible and in view of Lemma~\ref{freq-3}, this contradicts the complete  monotonicity of  $\left\{\frac{1}{p(m_0,n)}\right\}_{n \in \mathbb Z_+}.$
\begin{case}
    $b_1<0$
\end{case}
     Choose  $m_0 \in \mathbb{Z}_+$ such that $b_1m_0+b_2<0.$ Since $\left\{\frac{1}{p(m_0,n)}\right\}_{n \in \mathbb Z_+}$ is a complete monotone sequence, this contradicts Lemma~\ref{freq-3}.
    
    Hence $b_1>0.$ This completes the proof.
\end{proof}}
The following lemma provides necessary conditions for a class of polynomials in two variables whose reciprocal is joint completely monotone.   
\begin{lemma} \label{prop-n-condition}
Let $p$ be a polynomial in two variables given by  $p(x,y)=q(x)+r(x)y+s(x)y^2,$ where $q,r$ and $s$ are polynomials in one variable. Assume that $ p(m,n) \neq 0$ for every $m,n \in \mathbb{Z}_+.$ If the net $\{1/p(m,n)\}_{m,n\in \mathbb{Z}_+}$ is joint completely monotone, then  
\beq \label{p1-sqrt}
&& 4q(m)s(m) \Le r^2(m), \quad m \in \mathbb{Z}_+, \\
&& \deg (q)+\deg (s) \Le 2 \deg (r). \label{inq-new}
\eeq 
\end{lemma}
\begin{proof} Assume that $\left\{\frac{1}{p(m,n)}\right\}_{m,n\in \mathbb{Z}_+}$ is joint completely monotone. 
As noted earlier, $\left\{\frac{1}{p(m,n)}\right\}_{m,n\in \mathbb{Z}_+}$ is  separate completely monotone. Therefore, by Lemma~\ref{freq-3},  for any $m \in \mathbb{Z}_+,$  the roots of $p(m,.)$ are real numbers. Thus, we can apply the formula for the roots of a quadratic equation to obtain \eqref{p1-sqrt}. Note that \eqref{p1-sqrt} yields \eqref{inq-new}.
\end{proof}
 We need the following in the proof of the necessity part of  Theorem~\ref{MT}.
\begin{lemma}  \label{Ext-cm}
  Let $p$ be a polynomial in two variables given by $p(x,y)=a(x)+b(x)y+y^2,$ where $a$ and $b$ are polynomials in one variable. Assume that $ p(m,n) \neq 0$ and $b^2(m)\neq 4a(m)$ for every $m,n\in \mathbb{Z}_+$. Let $\left\{\frac{1}{p(m,n)}\right\}_{m,n \in \mathbb{Z}_+}$ be a joint completely monotone net.  Then, for any positive real numbers $\alpha$ and $\beta,$ $\left\{\frac{1}{p(m,\alpha m +\beta)}\right\}_{m \in \mathbb{Z}_+}$ is a completely monotone sequence.
\end{lemma}
 \begin{proof}
By Lemma~\ref{prop-n-condition} and the assumption that $b^2(m)\neq 4a(m)$ for every $m \in \mathbb{Z}_+,$
$$b^2(m)-4a(m)> 0, \quad m \in \mathbb Z_+.$$
    Also, for $m\in \mathbb Z_+$ and $y \in \mathbb R_+,$   $$p(m,y)=(y+r_1(m))(y+r_2(m)),$$ where $r_1$ and $r_2$ are given by 
    \beqn
    r_1(m)=\frac{b(m)+\sqrt{b^2(m)-4a(m)}}{2}, \quad r_2(m)=\frac{b(m)-\sqrt{b^2(m)-4a(m)}}{2}.
    \eeqn
    Note that for every $m \in \mathbb Z_+$ and $y \in \mathbb R_+,$
    \beqn
    \frac{1}{p(m,y)}&=&\frac{1}{(y+r_1(m))(y+r_2(m))} \\
    &=&\frac{1}{r_2(m)-r_1(m)}\left(\frac{1}{y+r_1(m)}-\frac{1}{y+r_2(m)}\right)\\
    &=&\int_0^1t^y\left(\frac{t^{r_1(m)-1}-t^{r_2(m)-1}}{r_2(m)-r_1(m)}\right)dt.
    \eeqn
    Therefore,
    \beq\label{t-w-int}
    \frac{1}{p(m,y)}=\int_0^1t^yw_m(t)dt, \quad m\in \mathbb Z_+,\,y \in \mathbb R_+,
    \eeq
    where $w_m$ is given by
    \beqn
    w_m(t)=\frac{t^{r_1(m)-1}-t^{r_2(m)-1}}{r_2(m)-r_1(m)}, \quad t \in (0,1).
    \eeqn
    Since the net $\left\{\frac{1}{p(m,n)}\right\}_{m,n \in  \mathbb Z_+}$ is joint completely monotone, by \eqref{t-w-int}
    \beqn
    (-1)^{i}\triangle_1^{i}\frac{1}{p(m,n)}=\int_{[0, 1]}t^n(-1)^{i}\triangle_1^{i}w_m(t)dt \Ge 0, \quad i,m \in \mathbb Z_+.
    \eeqn
    This implies for every $i,m \in \mathbb Z_+$ and $t \in (0,1),$
    $$(-1)^{i}\triangle_1^{i}w_m(t) \Ge 0.$$
    Therefore, for each $t\in (0,1),$ $\{w_{m}(t)\}_{m\in \mathbb Z_+}$
    is completely monotone.
    Let $\alpha$ and $\beta$ be positive real numbers. Note that for every $t\in (0,1),$ $\{t^{\alpha m+\beta}\}_{m\in \mathbb Z_+}$ is a completely monotone sequence. By \cite[Lemma 8.2.1(v)]{BCR1984} , for every $t \in (0,1),$ we have 
    \beqn
    (-1)^{i}\triangle ^{i}t^{\alpha m +\beta}w_m(t) \Ge 0, \quad m \in \mathbb Z_+.
    \eeqn
    This combined with \eqref{t-w-int}, yields
    \beqn
    (-1)^{i}\triangle^{i}\frac{1}{p(m,\alpha m+\beta)}&=&\int_0^1(-1)^{i}\triangle^{i}t^{\alpha m +\beta}w_m(t)dt \Ge 0, \quad m \in \mathbb Z_+.
    \eeqn
    This shows that $\left\{\frac{1}{p{(m,\alpha m + \beta)}}\right\}_{m\in \mathbb{Z}_+}$ is a completely monotone sequence. 
\end{proof}

\begin{proof}[Proof of Theorem~\ref{MT}] Assume that the net 
$\left\{\frac{1}{p(m,n)}\right\}_{m,n \in \mathbb{Z}_+}$ is joint completely monotone. 
A routine calculation shows that 
\beq \label{b-square-4a}
&& b^2(m)-4a(m) \\ \notag 
&=& (b_0^2-4a_0)m^2+(2b_0^2b_1-4a_0(a_1+a_2))m+b_0^2b_1^2-4a_0a_1a_2,
\eeq
which, by \eqref{p1-sqrt} (applied to $q=a$, $r=b$ and $s=1$), is nonnegative for every $m \in \mathbb Z_+$. 
It follows that $b_0^2-4a_0 \Ge 0$ and $b_0^2b_1^2-4a_0a_1 \Ge 0.$ By  Lemma~\ref{pos-(2,2)},  $a_0,a_1,a_2, b_0,b_1> 0.$
 
Before we prove the necessity part, 
consider the polynomials given by
\beq \label{f-g}
f(m)= \frac{b(m)}{2}, ~g(m)=\frac{\sqrt{b^2(m)-4a(m)}}{2}, \quad m \in \mathbb Z_+, 
\eeq
($g$ is real-valued since $b^2 \Ge 4a$) 
and note that 
\beq \label{1byp-f-g}
\frac{1}{p(m,n)}=\frac{1}{(n+f(m))^2-g^2(m)}, \quad m,n \in \mathbb{Z}_+.
\eeq
We will divide the verification of \eqref{Impeq} into the following cases.
\begin{case}
$\deg b^2-4a \Le 1$
\end{case}
If $\deg b^2-4a  =0,$ then by \eqref{b-square-4a}, $b_0^2=4a_0$ and $2b_1=a_1+a_2$, and hence \eqref{Impeq} holds. 
If possible, then assume that $b^2(m)-4a(m)$ is a linear polynomial. By \eqref{b-square-4a}, 
$b_0^2=4a_0,$
and hence for every $m \in \mathbb Z_+,$
\beq \label{b-0-square}
b^2(m)-4a(m)=b_0^2\big(2b_1-(a_1+a_2)\big)m+b_0^2(b_1^2-a_1a_2)=c_0m+c_1,
\eeq
where $c_0=b_0^2\big(2b_1-(a_1+a_2)\big)$ and $c_1=b_0^2(b_1^2-a_1a_2).$ Since $b^2(m)-4a(m)$ is a nonnegative linear polynomial (see \eqref{p1-sqrt}), we have 
\beq \label{a1plusa2}
a_1+a_2 < 2b_1.
\eeq 
A routine calculation using \eqref{f-g} and \eqref{b-0-square} shows that for $m,n \in \mathbb{Z}_+,$
\beqn
 p(m, n) &\overset{\eqref{1byp-f-g}}=&
 (n+f(m))^2-g^2(m) \\
 &=& \Big(\frac{b_0m}
 {2}+
 \frac{\frac{b_0^2b_1}{2}+b_0n-\frac{c_0}{4} }{b_0}\Big)^2
 +b_0\Big(b_1-\frac{(a_1+a_2)}{2}\Big)n \\ &+& \frac{b_0^2b_1^2}{4}
 -\frac{\Big(\frac{b_0^2b_1^2}{2}-\frac{c_0}{4}\Big)^2}{b_0^2}.
 \eeqn
    Since $a_1+a_2 < 2b_1$ (see \eqref{a1plusa2}) and  $b_0 >0,$ we note that there exists $n_0 \in \mathbb Z_+$ such that  
    \beqn
   \frac{\Big(\frac{b_0^2b_1^2}{2}-\frac{c_0}{4}\Big)^2}{b_0^2} - \frac{b_0^2b_1^2}{4} < b_0\Big(b_1-\frac{(a_1+a_2)}{2}\Big)n, \quad n \Ge n_0.
 \eeqn 
It follows that  $p(m,n_0)$ is irreducible in $m$, and hence $\left\{\frac{1}{p(m,n)}\right\}_{(m,n)\in \mathbb{Z}_+^2}$ is not separate completely monotone. Hence $\deg b^2-4a  =0,$ which completes the proof in this case. 

\begin{case}
$b^2(x)-4a(x)$ is a quadratic polynomial 
\end{case}
Note that by \eqref{p1-sqrt} and \eqref{b-square-4a}, $b_0^2-4a_0 > 0.$ It is easy to see using \eqref{f-g} that for every $m \in \mathbb{Z}_+,$ 
 \beqn 
 g^2(m)  &=&(c_0m+c_1)^2+c_2,
  \eeqn
  where $c_0, c_1, c_2$ are given by 
  \beq \label{c2-inq} \notag
  c_0&=&\frac{\sqrt{b_0^2-4a_0}}{2}, \quad \notag
 c_1= \frac{b_0^2b_1-2a_0(a_1+a_2)}{2\sqrt{b_0^2-4a_0}},\\
    c_2&=&-\frac{a_0(a_2-a_1)^2-b_0^2(b_1-a_1)(a_2-b_1)}{b_0^2-4a_0}.
  \eeq

Now, we choose a very large $\alpha_0 \in \mathbb{Z}_+ $   such that $c_0 \alpha_0 +c_1 >0$ and $b^2(m+\alpha_0) \neq 4a(m+\alpha_0)$ for every $m\in \mathbb{Z}_+.$
 We also choose a very large natural number, say $N_0>1,$ such that
   \beq\label{l_1-l_2}
   l_1:=N_0c_0-\frac{b_0}{2}> 0, \quad l_2:=N_0(c_0\alpha_0+c_1)-\frac{b_0\alpha_0}{2}-\frac{b_0b_1}{2}>0.
   \eeq
   Take $n=l_1m+l_2$ and consider 
     \beqn
   && \frac{1}{p(m+\alpha_0,l_1m+l_2)}\\&\overset{\eqref{1byp-f-g}}=& \frac{1}{(l_1m+l_2+f(m+\alpha_0))^2-g^2(m+\alpha_0)}\\&\overset{\eqref{f-g},\eqref{l_1-l_2}}=&\frac{1}{(N_0c_0m+N_0(c_0\alpha_0+c_1))^2-(c_0m+c_0\alpha_0+c_1)^2-c_2}\\
     &=&\frac{1}{(N_0^2-1)(c_0m+c_0\alpha_0+c_1)^2-c_2}.
     \eeqn
   Assume that \eqref{Impeq} does not hold. By \eqref{c2-inq}, we obtain $c_2<0.$ Therefore, the polynomial $(N_0^2-1)(c_0m+c_0 \alpha_0+c_1)^2-c_2$ is irreducible in $m.$ One may see, using Lemma~\ref{freq-3} that the sequence $\left\{\frac{1}{p(m+ \alpha_0, l_1m+l_2)}\right\}_{m \in \mathbb{Z}_+}$ is not completely monotone. This is not possible in view of Lemma~\ref{Ext-cm} and Remark~\ref{remark-cm-gen}.
     
    This completes the proof of the necessity part.
    
We will divide the verification of the sufficiency part into several cases.
\begin{casen}
$b^2(m)-4a(m)$ is a constant 
\end{casen}
By \eqref{b-square-4a}, we have $2b_1=a_1+a_2$ and $b_0^2 = 4a_0.$ This implies  for every $m\in \mathbb Z_+,$ 
 \beqn
 b^2(m)-4a(m)\overset{\eqref{b-square-4a}}=b_0^2b_1^2-4a_0a_1a_2  = a_0(a_1-a_2)^2 \Ge 0.
 \eeqn
It follows that for $m,n\in \mathbb{Z}_+,$  
\begin{multline}
p(m,n) \overset{\eqref{1byp-f-g}}= \\ \notag \Big(n+\frac{b_0}{2}m+\frac{b_0b_1}{2}+\frac{\sqrt{b_0^2b_1^2-4a_0a_1a_2}}{2}\Big)\Big(n+\frac{b_0}{2}m+\frac{b_0b_1}{2}-\frac{\sqrt{b_0^2b_1^2-4a_0a_1a_2}}{2}\Big).
\end{multline}
Clearly, $\frac{b_0b_1}{2}+\frac{\sqrt{b_0^2b_1^2-4a_0a_1a_2}}{2} \Ge 0$ and $\frac{b_0b_1}{2}-\frac{\sqrt{b_0^2b_1^2-4a_0a_1a_2}}{2} \Ge 0.$ In this case, $\{1/p(m,n)\}_{m,n\in \mathbb{Z}_+}$ is a joint completely monotone net since it is the product of two joint completely monotone net (see \cite[Lemma 8.2.1(v)]{BCR1984}). 

\begin{casen}
$b^2(m)-4a(m)$ is a linear polynomial
\end{casen}
 Note that from \eqref{Impeq},
    \beqn
    a_0(a_2-a_1)^2  \Le  b_0^2(b_1-a_1)(a_2-b_1).
    \eeqn
Since $b^2(m)-4a(m)$ is a linear polynomial, we have $4a_0=b_0^2,$ and hence      
    \beqn
    (a_2-b_1+b_1-a_1)^2  \Le  4(b_1-a_1)(a_2-b_1).
    \eeqn
    It now follows that
    \beqn 
    (a_2-b_1)^2+(b_1-a_1)^2+2(a_2-b_1)(b_1-a_1) \Le  4(b_1-a_1)(a_2-b_1),
    \eeqn
    which clearly yields 
  $ (a_2-2b_1+a_1)^2 \Le  0,$ or equivalently,  $a_1+a_2=2b_1.$ Thus, this case reduces to that of (1). Therefore, the net $\left\{\frac{1}{p(m,n)}\right\}_{m,n\in \mathbb{Z}_+}$ is joint completely monotone.
  
\begin{casen}   $b^2(x)-4a(x)$ is a quadratic polynomial 
 \end{casen}
 Note that $b_0^2 > 4a_0.$ For every $m \in \mathbb{Z}_+,$ we already noted that 
 \beqn
 g^2(m)&=& (c_0m+c_1)^2+c_2,
  \eeqn
  where $c_0, c_1, c_2$ are given by 
  \beq \notag
  c_0&=&\frac{\sqrt{b_0^2-4a_0}}{2}, \quad
 c_1=\frac{b_0^2b_1-2a_0(a_1+a_2)}{2\sqrt{b_0^2-4a_0}},\\ \label{c2}
  c_2&=&
  -\frac{a_0(a_2-a_1)^2-b_0^2(b_1-a_1)(a_2-b_1)}{b_0^2-4a_0}.
  \eeq
Also note that by \eqref{1byp-f-g}, for $m,n \in \mathbb{Z}_+,$
  \beqn
  \frac{1}{p(m, n)} &=& \frac{1}{(n+b(m)/2)^2-((c_0m+c_1)^2+c_2)} \\
  &=& \frac{1}{(n+\frac{b(m)}{2}+c_0m+c_1)(n+\frac{b(m)}{2}-c_0m-c_1)-c_2}\\
   &=& \frac{1}{p_1(m,n)p_2(m,n)-c_2},
   \eeqn
   where $p_1$ and $p_2$ are given by
   \beqn
   p_1(m,n):=n+(b_0/2+c_0)m+(b_0b_1/2+c_1),  \quad m, n \in \mathbb{Z}_+ , \\ 
 p_2(m,n):=n+(b_0/2-c_0)m+(b_0b_1/2-c_1),  \quad m,n \in \mathbb{Z}_+. 
 \eeqn
By  \eqref{Impeq} and \eqref{c2}, $c_2 \Ge 0.$ If $c_1 \Ge 0,$ then  $b_0b_1/2+c_1 \Ge 0,$ and since 
\beqn
(b_0b_1/2+c_1)(b_0b_1/2-c_1) = p_1(0,0)p_2(0,0) > c_2\Ge 0,
\eeqn 
we must have $b_0b_1/2-c_1 > 0.$ 
Similarly, if $c_1<0,$ then $b_0b_1/2-c_1 > 0,$ and hence $b_0b_1/2+c_1 > 0.$ 
Thus, for any real value of $c_1,$  $\{1/p_1(m,n)\}_{m,n\in \mathbb{Z}_+}$  and $\{1/p_2(m,n)\}_{m,n\in \mathbb{Z}_+}$ are joint completely monotone nets (see \cite[Theorem~3.1]{ACN}). Note that for $m,n \in \mathbb{Z}_+,$ $p(m,n)=p_1(m,n)p_2(m,n)-c_2> 0.$ Thus, we have 
\beq \label{lessthan1}
 \frac{c_2}{p_1(m,n)p_2(m,n)}<1, \quad  m,n \in \mathbb{Z}_+.
\eeq
Therefore, for all $m,n \in \mathbb{Z_+},$
\beqn
\frac{1}{p_1(m,n)p_2(m,n)-c_2}&=&\frac{1}{p_1(m,n)p_2(m,n)}\Big(\frac{1}{1-\frac{c_2}{p_1(m,n)p_2(m,n)}}\Big)\\
&\overset{\eqref{lessthan1}}=&\sum_{k=0}^{\infty}\frac{c_2^k}{(p_1(m,n)p_2(m,n))^{k+1}}.
\eeqn
Since, for each $k\in \mathbb{Z}_+$, $\left\{\frac{c_2^k} {(p_1(m,n)p_2(m,n))^{k+1}}\right\}_{m,n \in \mathbb{Z}_+}$ is a joint completely monotone net, the finite sum $\left\{\sum_{k=0}^\ell\frac{c_2^k}{(p_1(m,n)p_2(m,n))^{k+1}}\right\}_{m,n\in \mathbb{Z}_+},$ where $ \ell \in \mathbb Z_+,$ is also joint completely monotone. Since the limit of the joint completely monotone net is joint completely monotone (see \cite[p. 130]{BCR1984}), we conclude that the net $\left\{\sum_{k=0}^{\infty}\frac{c_2^k}{(p_1(m,n)p_2(m,n))^{k+1}}\right\}_{m,n \in \mathbb{Z}_+}$ is joint completely monotone.
This completes the proof of the sufficiency part.
\end{proof}

 \section{The Cauchy dual subnormality problem}
In this section, we present a proof of the Theorem~\ref{CDSP-T3I}. We begin with the following proposition which is a consequence  of \cite[Proposition~4.6]{ACN}.

\begin{proposition} \label{3-iso-wt-char}
For a weighted $2$-shift $\mathscr W : \{w^{(j)}_\alpha\},$ the following statements are valid$:$
 \begin{enumerate}
\item[$\mathrm{(i)}$] $\mathscr W$ is a toral $3$-isometry if and only if for $\alpha=(\alpha_1,\alpha_2)\in \mathbb Z_+^2,$
\allowdisplaybreaks 
 \beqn 
 \|\mathscr W^\alpha e_0\|^2  = 1+a_1\alpha_1  +a_2\,\alpha_1^2+ (b_1\alpha_1+b_2) \alpha_2+ c_1\,\alpha_2^2,
 \eeqn
 where $a_1, a_2,b_1,b_2$ and $c_1$ are as follows$:$
 \beqn\label{coef-expre}
a_1=\rho_{10}-\frac{\rho_{20}}{2},\, 
 a_2= \frac{\rho_{20}}{2},\, b_1=\rho_{11},\, 
 b_2 = \rho_{01}-\frac{\rho_{02}}{2},\,
 c_1 = \frac{\rho_{02}}{2}.
 \eeqn
 \item[$\mathrm{(ii)}$]
$\mathscr W$ is a toral $3$-isometry with $\mathscr W_2$ being a $2$-isometry if and only if for $\alpha=(\alpha_1, \alpha_2)\in \mathbb Z_+^2,$
 \beqn 
 \|\mathscr W^\alpha e_0\|^2 
= 1+a\,\alpha_1+b\,\alpha_1^2 + (c +d\, \alpha_1\,) \,\alpha_2,
\eeqn
where $a,b,c$ and $d$ are as follows$:$
\allowdisplaybreaks
\beqn
a = \rho_{10}-\frac{\rho_{20}}{2}, \quad 
b =\frac{\rho_{20}}{2}, \quad 
c = \rho_{01}, \quad 
d = \rho_{11}.
\eeqn
\end{enumerate}

 \end{proposition}
 \begin{proof} This follows from \cite[Proposition~4.6 ]{ACN} (the case of $m = 3$).
 \end{proof}
The following proposition reveals a relation between joint subnormality of
the toral Cauchy dual and joint complete monotonicity.
 \begin{proposition}\label{subnormal-reciprocal}
    Let $\mathscr W$ be a torally
expansive weighted $n$-shift and let $\mathscr W^{\mathfrak{t}}$ be its toral Cauchy dual. Then $\mathscr W^{\mathfrak{t}}$ is jointly subnormal if and only if the net $\left\{ \frac{1}{ \|\mathscr W^{\alpha} e_0\|^2}\right\}_{ \alpha  \in \mathbb Z^n_+}$ is joint completely monotone.
 \end{proposition}
 \begin{proof} Recall that for a torally contractive weighted $n$-shift $U$, the following holds$:$
  \begin{align} \label{subnor-if-r}
   \begin{minipage}{67ex}
\text{$ U$ is jointly subnormal if and only if $\{\| U^\alpha e_0\|^2\}_{\alpha \in \mathbb Z^n_+}$ is a joint}
\\
\text{completely monotone net.} 
\end{minipage}
 \end{align}
 This fact can be deduced from \cite[Theorem 4.4]{Athavale1987}, together with \eqref{W-beta} 
 (see also the discussion prior to \cite[Eqn~(E)]{Athavale2001}). By  \eqref{dual-action} and the discussion following it,  $\mathscr W^{\mathfrak t}$
is a commuting weighted $n$-shift. Since $\mathscr W$ is torally expansive, routine calculations show that $\mathscr W^t$ is torally contractive. This, combined with \eqref{reciprocal} and \eqref{subnor-if-r}, completes the proof of the proposition.
\end{proof}
{
We now present a solution to the CDSP for torally expansive toral $3$-isometric weighted $2$-shifts.

\begin{proof}[Proof of Theorem~\ref{CDSP-T3I}]
Recall that by Proposition~\ref{3-iso-wt-char}(i), for $\alpha=(\alpha_1, \alpha_2) \in \mathbb Z^2_+,$
\beq \label{pos-1-coef}
 \|\mathscr W^\alpha e_0\|^2=1+a_1\alpha_1+a_2\alpha_1^2+(b_1\alpha_1+b_2)\alpha_2+c_1\alpha_2^2,
\eeq
where $a_1,a_2,b_1,b_2$ and $c_1$ are given by
\beqn
a_1=\rho_{10}-\frac{\rho_{20}}{2}, \quad 
 a_2= \frac{\rho_{20}}{2}, \quad  b_1=\rho_{11}, \quad 
 b_2 = \rho_{01}-\frac{\rho_{02}}{2}, \quad 
 c_1 = \frac{\rho_{02}}{2}.
 \eeqn
 
It is clear from \eqref{pos-1-coef} that $a_2\Ge0$ and $c_1\Ge 0.$ Let $p(x,y)=1+a_1x+a_2x^2+(b_1x+b_2)y+c_1y^2.$ By \eqref{reciprocal}, $\alpha= (\alpha_1, \alpha_2) \in \mathbb{Z}_+^2,$
\beqn
 \|(\mathscr W^{\mathfrak{t}})^\alpha e_0\|^2  = 
\frac{1}{p(\alpha_1,\alpha_2)}.
\eeqn
We record that for every $\alpha_1,\alpha_2 \in \mathbb{Z}_+,$ $p(\alpha_1,\alpha_2)>0.$

(a) This has been proved in \cite[Theorem~4.9]{ACN}.

(b)  Since $\mathscr W_1$ is not a $2$-isometry, $a_2>0$. We divide the proof in two cases.
\begin{case-new}
    $\mathscr W_2$ is a $2$-isometry.
\end{case-new}
 Note that $c_1=0.$ By Corollary~\ref{Imp-coro}(i), $b_1,b_2\Ge 0.$ For the necessity part, assume that $\mathscr W^{\mathfrak{t}}$ is jointly subnormal.
By Proposition~\ref{subnormal-reciprocal},
$\left\{\frac{1}{p(\alpha_1,\alpha_2)}\right\}_{\alpha_1,\alpha_2 \in \mathbb Z_+}$ is a joint completely monotone net. Hence, it is separate completely monotone. Applying Lemma~\ref{freq-3} to the polynomial $p(x,0),$ we obtain $a_1^2 \Ge 4a_2$ and $a_1>0.$ This yields $(2\rho_{10}-\rho_{20})^2\Ge 8\rho_{20}$ and $2\rho_{10}>{\rho_{20}}.$ In view of Corollary~\ref{Imp-coro}(ii), either $b_1,b_2=0$ or $b_1,b_2>0.$ If $b_1,b_2=0,$ we are done. If $b_1,b_2>0,$ we apply Theorem~\ref{bi-deg-2-1-second} to complete the proof of the necessity part.

To see the sufficiency part, assume that $(2\rho_{10}-\rho_{20})^2\Ge 8\rho_{20}$ and $2\rho_{10}>{\rho_{20}}.$ This is equivalent to $a_1>0$ and $a_1^2\Ge4a_2.$  Since $a_1^2 \Ge  4a_2,$ we have
\beqn 
p(x,y)&=&a_2\Big(x+\frac{a_1-\sqrt{a_1^2-4a_2}}{2a_2}\Big)\Big(x+\frac{a_1+\sqrt{a_1^2-4a_2}}{2a_2}\Big)\\
&&+(b_1x+b_2)y.
\eeqn
We consider the following two subcases to complete the proof of sufficiency part in this case.

\begin{subcase}
 $\rho_{11}=0,\rho_{01}=0,{\rho_{02}}=0$ (equivalently, $b_1=0,b_2=0$).
 Note that $\left\{\frac{1}{p(\alpha_1,\alpha_2)}\right\}_{\alpha_1,\alpha_2\in \mathbb{Z}_+}$ is the product of two completely monotone sequences and hence it is joint completely monotone. By Proposition \ref{subnormal-reciprocal}, $\mathscr W^{\mathfrak{t}}$ is jointly subnormal.
 \end{subcase}
\begin{subcase}
   $\rho_{11}>0,2\rho_{01}>\rho_{02}$ and $$(\rho_{20}\rho_2-\rho_{11}\rho_1)^2  \Le (4\rho_{11}^2-{\rho_{20}\rho_{02}}) (\frac{\rho^2_1}{4}-2\rho_{20})
   $$
where $\rho_1=2\rho_{10}-\rho_{20}$ and $\rho_2=2\rho_{01}-{\rho_{02}}$ (equivalently, $b_1>0,b_2>0,$ and  $({2a_2b_2}-a_1b_1)^2\Le b_1^2 (a_1^2-4a_2)$).
 A routine calculation yields that
\beqn 
\frac{a_1-\sqrt{a_1^2-4a_2}}{2a_2} \Le \frac{b_2}{b_1} \Le \frac{a_1+\sqrt{a_1^2-4a_2}}{2a_2}.
\eeqn
 We now apply Theorem~\ref{bi-deg-2-1-second} and Proposition~\ref{subnormal-reciprocal} to complete the proof of the sufficiency part in this case.
\end{subcase}
\begin{case-new}
$\mathscr W_2$ is not a $2$-isometry.
\end{case-new}
Note that $c_1>0$.
By Proposition~\ref{subnormal-reciprocal},
$\mathscr W^{\mathfrak{t}}$ is jointly subnormal if and only if the net $\left\{\frac{1}{p(\alpha_1,\alpha_2)}\right\}_{\alpha_1,\alpha_2 \in \mathbb Z_+}$ is joint completely monotone or equivalently  
\beqn \label{mod-SCM}
\left\{\frac{1}{1/c_1+(a_1/c_1)\alpha_1+(a_2/c_1)\alpha_1^2+((b_1\alpha_1+b_2)/c_1)\alpha_2+\alpha_2^2}\right\}_{\alpha_1,\alpha_2 \in \mathbb Z_+}
\eeqn
is a joint completely monotone net. The proof of the necessity part now follows from Lemmas~\ref{freq-3},~\ref{pos-(2,2)} and Theorem~\ref{MT}. Sufficiency follows from Theorem~\ref{MT}.

(c) This follows from part (b), by interchanging the role of $\mathscr W_1$ and $\mathscr W_2.$
\end{proof} 
}


\noindent
\textit{Acknowledgements.} 
I would like to express my sincere gratitude to Prof. Akash Anand and Prof. Sameer Chavan for their invaluable guidance and support throughout the preparation of this paper.

{}
\end{document}